\documentclass[12pt,reqno]{amsart}

\numberwithin{equation}{section}

\usepackage{amssymb}
\usepackage{fullpage} 
\usepackage{xcolor}

\usepackage{hyperref} 
\definecolor{mycitecolor}{rgb}{1,0,0}
\definecolor{mylinkcolor}{rgb}{0.66,0,0} 
\definecolor{myurlcolor}{rgb}{0.33,0,0}
\hypersetup{colorlinks=true,citecolor=mycitecolor,linkcolor=mylinkcolor,urlcolor=myurlcolor,breaklinks=true}

\definecolor{labelkey}{rgb}{0,0,1} 
\definecolor{refkey}{rgb}{0,0,0.5} 

\usepackage{stmaryrd} 

\theoremstyle{plain}
\newtheorem{theorem}{Theorem}[section]

\newtheorem{corollary}[theorem]{Corollary}

\newtheorem{lemma}[theorem]{Lemma}

\newtheorem{proposition}[theorem]{Proposition}

\theoremstyle{definition}

\newtheorem{remark}[theorem]{Remark}

\newtheorem{example}[theorem]{Example}

\numberwithin{equation}{section}

\def\ldiv{\backslash}
\def\rdiv{/}

\def\comm#1{\llbracket#1\rrbracket} 
\def\assoc#1{\llbracket#1\rrbracket} 

\def\nuc#1{\mathrm{Nuc}(#1)}

\def\aut#1{{\mathrm{Aut}(#1)}}

\def\mlt#1{\mathrm{Mlt}(#1)}

\def\inn#1{\mathrm{Inn}(#1)}

\def\supp{\mathrm{supp}}
\def\lesupp{\preceq}
\def\lsupp{\prec}
\def\lenorm{\le_{\mathrm{norm}}}
\def\lelex{\le_{\mathrm{lex}}}

\def\weight{\omega}

\title{Supernilpotent groups and $3$-supernilpotent loops}

\author{David Stanovsk\'y}
\author{Petr Vojt\v echovsk\'y}

\address{Department of Algebra, Faculty of Mathematics and Physics, Charles University, Soko\-lovsk\'a 83, 18675 Prague, Czech Republic}
\email{stanovsk@karlin.mff.cuni.cz}

\address{Department of Mathematics, University of Denver, 2390 S. York St., Denver, Colorado 80208, U.S.A.}
\email{petr.vojtechovsky@du.edu}

\thanks{The paper is written within the framework of the cooperation grant LTAUSA19070. P. Vojt\v echovsk\'y partially supported by the Simons Foundation Mathematics and Physical Sciences Collaboration Grant for Mathematicians no.~855097 and by the PROF grant of the University of Denver.}



\date{\today}

\begin{document}

\begin{abstract}
We find a short equational basis for the variety of $3$-supernilpotent loops. We also present a conceptually simple proof that $k$-nilpotence and $k$-supernilpotence are equivalent for groups. Connections between $3$-supernilpotent loops, Moufang loops, code loops, automorphic loops and AIM loops are explored.
\end{abstract}

\maketitle

\section{Introduction}

The classical approach to nilpotence of algebraic structures (shortly, algebras) is recursive: define the center, then a central series, and let the class of nilpotence be the length of a shortest central series.

In groups, the center consists of all elements that commute with everything. In loops, the center consists of all elements that commute \emph{and associate} with everything \cite{Br-book}.

In 1970s, universal algebraists found a suitable syntactic condition to define the center of any algebra \cite{FM}. However, many characteristic properties of nilpotent groups do not carry over to this more general setting. Perhaps most important among such properties is the fact that there are finite nilpotent algebras, loops in particular, which do not admit a direct decomposition into $p$-primary components.

This issue was addressed recently in a novel way \cite{AM} that is based on another fundamental property: the limited essential arity of absorbing polynomial operations. An algebra is called \emph{$k$-supernilpotent} if all absorbing polynomials of arity bigger than $k$ are constant. See Section \ref{Sc:Absorbing} for a more detailed definition.

For groups, $k$-supernilpotence coincides with $k$-nilpotence \cite{AE}. Under mild universal algebraic assumptions (which cover groups and loops), supernilpotence implies nilpotence \cite{AM,Moo} and a finite algebra is supernilpotent if and only if it is a direct product of nilpotent algebras of prime power order \cite[Lemma 7.6]{AM}. Nilpotence does not imply supernilpotence in loops (Example \ref{Ex:6}).

The study of supernilpotence in loops was initiated in \cite{SS}. In the present paper, we address the problem of equational axiomatization of $k$-supernilpotent loops.

If $\mathcal V$ is a variety with a Mal'tsev term, then $k$-supernilpotent algebras from $\mathcal V$ form a subvariety $\mathcal V_k$ \cite[Theorem 4.2]{AMO}. An explicit set of identities axiomatizing $\mathcal V_k$ relative to an axiomatization of $\mathcal V$ is known \cite{AMO}, but it is rather complicated, infinite, and unnatural in the context of groups and loops. Given a particular variety $\mathcal V$, it is not clear from the general result of \cite{AMO} whether a finite basis of $\mathcal V_k$ exists. In groups, the question is of course settled by the equivalence of $k$-supernilpotence and $k$-nilpotence. In the present paper, we aim at a finite equational basis of $k$-supernilpotent loops based on elementwise commutators and associators, in the spirit of the commutator identity $[[[x_1,x_2],\ldots],x_{k+1}]=1$ that axiomatizes $k$-(super)nilpotent groups. We succeed for $k\leq 3$ (Theorems \ref{Pr:SupernilpotentLoops12} and \ref{Th:MainLoops}).

The paper has two essential parts. In Section \ref{sec:groups} we present a conceptually simple proof of the equivalence of $k$-nilpotence and $k$-supernilpotence for groups (Theorem \ref{Th:MainGroups}). Then we focus on loops. In Section \ref{sec:loops_lemmas} we introduce commutators and associators in loops and establish several technical properties. Section \ref{sec:sn in loops} contains our main result, a small equational basis for $3$-supernilpotent loops (Theorem \ref{Th:MainLoops}). In Section \ref{Sc:Consequences} we explore consequences of and connections to the main result. In Subsection \ref{Ss:Inner} we prove that any two standard generators of inner mapping groups of $3$-supernilpotent loops commute, with the exception of two conjugations. We also point out connections to automorphic loops and AIM loops. In Subsection \ref{Ss:Moufang} we show that $2$-nilpotent Moufang loops (which include code loops) are $3$-supernilpotent. In Subsection \ref{Ss:ZN} we show how the class of supernilpotence is reduced modulo the center and modulo the nucleus (if the nucleus is normal) in $k$-supernilpotent loops. We also prove that the nucleus is always normal in $3$-supernilpotent loops. Finally, in Subsection \ref{Ss:Comp} we computationally identify all supernilpotent loops among very small nilpotent loops.

\section{Absorbing polynomials and supernilpotence}\label{Sc:Absorbing}

Let us fix a language $L$ containing a binary symbol $\cdot$ and a constant $1$. Throughout the paper, we have in mind two particular examples: groups in the language $\{\cdot,^{-1},1\}$ and loops in the language $\{\cdot,\ldiv,\rdiv,1\}$. While we often use additional symbols in terms and polynomials, such as commutators $[\_,\_]$, associators $[\_,\_,\_]$ and one-sided inverses $^\lambda,\,^\rho$ in loops, introducing these term definable symbols into the language does not change the class of (super)nilpotence.

Let $A$ be an algebra in the language $L$. By a \emph{polynomial} on $A$ we mean a term in the language $L\cup A$, where elements of $A$ are treated as (new) constants. Two polynomials on $A$ are called \emph{equivalent} if they induce the same functions on $A$.

The \emph{support} of a polynomial $p$ is the set $\supp(p) = \{x: x\text{ is a variable that occurs in }p\}$. Note that the support of $p$ is independent of the constants contained in $p$.

Let $e_1,\dots,e_n,e\in A$. An $n$-ary polynomial $p$ on $A$ is called \emph{absorbing at $(e_1,\dots,e_n)$ into $e$}, if $p(a_1,\dots,a_n)=e$ whenever $a_1,\dots,a_n\in A$ are such that $a_i=e_i$ for at least one $i$. Following \cite{AM}, an algebra $A$ is called \emph{$k$-supernilpotent} if all absorbing polynomials (at any tuple and into any element) in at least $k+1$ variables are constant.

There are several alternative definitions of $k$-supernilpotence, cf.~\cite{AM,Bu,Op}, however, we find the above definition to be most convenient for our purposes.

The following fact is our main tool to establish $k$-supernilpotence.

\begin{proposition}\label{Pr:MainTechnical}
Let $(A,\cdot,1,\ldots)$ be an algebra such that $1$ is the identity element with respect to the binary operation $\cdot$. Let $p$ be an $n$-ary polynomial on $A$ with $\supp(p)=\{x_1,\dots,x_n\}$ that is absorbing at $(e_1,\dots,e_n)$ into $1$. Assume that $p$ is equivalent to $\prod_{S\in\mathcal S}p_S$, where the product (in some order and in some parenthesizing) ranges over a subset $\mathcal S$ of proper subsets of $\{1,\dots,n\}$, and where every $p_S$ is a polynomial with $\supp(p_S)=\{x_i: i\in S\}$ that is absorbing at $(e_i:i\in S)$ into $1$. Then $p$ is constant.
\end{proposition}
\begin{proof}
The subsets of $\{1,\dots,n\}$ can be linearly ordered by extending the inclusion partial order. Let us denote by $\lesupp$ any such linear order.

For a contradiction, let $a_1,\dots,a_n\in A$ be such that $p(a_1,\dots,a_n)\ne 1$. Since $p$ is equivalent to  $\prod_{S\in\mathcal S}p_S$, there is a proper subset $S$ of $\{1,\dots,n\}$ such that $p_S(a_i:i\in S)\ne 1$ and $S$ is the least subset with this property with respect to the linear order $\lesupp$. Let us investigate $p(b_1,\dots,b_n)$, where $b_i=a_i$ if $i\in S$ and $b_i=e_i$ if $i\not\in S$.

Consider a proper subset $T\in\mathcal S$ of $\{1,\dots,n\}$. If $T\setminus S\ne\emptyset$ then $p_T(b_i:i\in T)=1$ because $p_T$ is absorbing at $(e_i:i\in T)$ and there is $i\in T$ with $b_i=e_i$. In all other situations we have $T\subseteq S$. If $T$ is a proper subset of $S$ then $T\lsupp S$ (since $\lesupp$ extends inclusion) and therefore $p_T(b_i:i\in T)=p_T(a_i:i\in T)=1$ by minimality of $S$.

It follows that $p(b_1,\dots,b_n)=p_S(b_i:i\in S)$. Now,  $p$ is absorbing at $(e_1,\dots,e_n)$, and $b_i=e_i$ for some $i$ because $S$ is a proper subset of $\{1,\dots,n\}$. Therefore $1=p(b_1,\dots,b_n)=p_S(b_i:i\in S)=p_S(a_i:i\in S)\ne 1$, a contradiction.
\end{proof}

It is easy to see that in loops we can without loss of generality consider only absorption at $(1,\dots,1)$ into $1$, cf.~\cite{SS}. In particular, a loop is $k$-supernilpotent if and only if all absorbing polynomials at $(1,\dots,1)$ into $1$ in at least $k+1$ variables are constant. Therefore, for the balance of the paper, an \emph{absorbing polynomial} will always mean an absorbing polynomial at $(1,\dots,1)$ into $1$.

\section{Supernilpotence in groups}\label{sec:groups}

Here we present a conceptually simple proof of $k$-nilpotence and $k$-supernilpotence for groups. We start with a classical result of Phillip Hall, cf. \cite[Theorem 2.53 and top of page 54]{Hall}. Denote by $[x,y]$ the standard commutator term $x^{-1}y^{-1}xy$ in groups.

Let $Y$ be a set of variables, their inverses and constants. Then \emph{complex commutator polynomials} and their \emph{weights} with respect to $Y$ are defined as follows: every $y\in Y$ is a complex commutator polynomial of weight $1$, and if $t_1$, $t_2$ are complex commutator polynomials of weights $k_1$ and $k_2$, respectively, then $[t_1,t_2]$ is a complex commutator polynomial of weight $k_1+k_2$. The weight of a complex commutator polynomial $t$ with respect to $Y$ will be denoted by $\weight_Y(t)$ or just by $\weight(t)$.

\emph{Simple commutator polynomials} form a subset of complex commutator polynomials, defined inductively as follows: every $y\in Y$ is a simple commutator polynomial (of weight $1$), and if $t$ is a simple commutator polynomial of weight $k$ and $y\in Y$ then $[t,y]$ is a simple commutator polynomial of weight $k+1$.

\emph{Complex commutators} (resp. \emph{simple commutators}) of weight $k$ in a group $G$ are then obtained by evaluating complex commutator polynomials (resp. simple commutator polynomials) of weight $k$ on $G$.

The usual definition of $k$-nilpotence for groups can be restated as follows: A group $G$ is \emph{$k$-nilpotent} if and only if all simple commutators of weight $k+1$ in $G$ are trivial.

\begin{theorem}[P.~Hall]\label{Th:Hall}
Let $G$ be a group and $k$ a positive integer. Then all simple commutators of weight $k$ on $G$ are trivial if and only if all complex commutators of weight $k$ on $G$ are trivial.
\end{theorem}

For a fixed set $X$ of variables and a fixed set $C$ of constants, we define the \emph{norm} of a complex commutator polynomial $t$ as
\begin{displaymath}
    N(t)=(\supp(t),\weight_{X^{\pm 1}\cup C}(t)).
\end{displaymath}
For instance, let $X=\{x_1,x_2\}$, $C=\{c\}$ and $t=[[x_1,x_2],[x_1^{-1},c]]$. Then $\supp(t)=\{x_1,x_2\}$, $\weight_{X^{\pm 1}\cup C}(t) = 4$ and therefore $N(t)=(\{x_1,x_2\},4)$. Let us emphasize that constants contribute to the weight $\weight_{X^{\pm 1}\cup C}(t)$ but not to $\supp(t)$.

We order norms lexicographically as follows. In the first component we employ any linear order that extends the inclusion partial order on the power set of $X$. In the second component we employ the usual linear order of natural numbers. The resulting lexicographic (linear) order on the norms will be denoted by $\lenorm$.

Our next goal is to rewrite products of complex commutator polynomials up to equivalence so that the norms of individual factors are nondecreasing. For a product of complex commutator polynomials $p=t_1\cdots t_\ell$ and for two norms $a>_{\mathrm{norm}} b$ we let
\begin{displaymath}
    d(p)_{a,b} = |\{(t_i,t_j):i<j,\,N(t_i)=a\text{ and } N(t_j)=b\}|.
\end{displaymath}
The \emph{disorder} of $p$ is then the sequence
\begin{displaymath}
    d(p) = (d(p)_{a,b}:a,b \text{ are norms satisfying }a>_{\mathrm{norm}}b ),
\end{displaymath}
where the norm pairs $(a,b)$ are ordered lexicographically by $\lenorm$ in an increasing fashion. Hence $(t_i,t_j)$ contributes to the disorder of $p$ if $t_i$ appears to the left of $t_j$ but is strictly larger in norm than $t_j$. We wish to achieve $d(p)=(0,0,\dots)$.

If the set of all possible norm pairs $(a,b)$ with $a>_{\mathrm{norm}} b$ is finite, say of size $m$, the disorder $d(p)$ can be identified with an element of $(\mathbb N^m,\lelex)$, a set lexicographically ordered by $\lelex$ with respect to the usual linear order on $\mathbb N$.

\begin{lemma}\label{Lm:Reorder}
Let $G$ be a $k$-nilpotent group and $p$ a polynomial on $G$ with support $X$. Then $p$ is equivalent to a product of complex commutator polynomials with respect to $X^{\pm 1}\cup G$ in which the norms of the factors are nondecreasing and the weight of every factor is at most $k$.
\end{lemma}
\begin{proof}
Let $\weight = \weight_{X^{\pm 1}\cup G}$ and denote $\lenorm$ by $\le$. By Theorem \ref{Th:Hall}, all complex commutators of weight more than $k$ vanish on $G$. Up to equivalence, we can therefore ignore complex commutator polynomials of weight more than $k$, the set of norms is then finite, and we view $d(p)$ as an element of $(\mathbb N^m,\lelex)$ for a suitable integer $m$.

Up to equivalence, we can write $p$ as a product of constants, variables and their inverses. In this way, $p$ is a product of (complex) commutator polynomials of weight $1$. Let us construct a sequence $p_1=p$, $p_2$, $p_3$, $\dots$ of products of complex commutator polynomials as follows. If $p_n$ is given and $d(p_n)=(0,0,\dots)$, stop. Else there are factors $u$, $v$ of $p_n$ such that $u$ is to the left of $v$ and $N(u)>N(v)$. In fact, there must be such consecutive factors, so let us assume that $u$, $v$ are also consecutive. We then replace $uv$ in $p_n$ with $vu[u,v]$ to obtain $p_{n+1}$, an equivalent product of complex commutator polynomials with one more factor. Clearly,
\begin{equation}\label{Eq:NormUp}
    N([u,v]) = (\supp(u)\cup\supp(v),\weight(u)+\weight(v))>\mathrm{max}\{N(u),N(v)\}.
\end{equation}
On account of $uv$ being replaced with $vu$, $d(p_{n+1})_{N(u),N(v)}$ is reduced by $1$ compared to $d(p_n)_{N(u),N(v)}$. Every other change to $d(p_n)$ is on account of $[u,v]$ and some other factor $w$. Now, if $N([u,v])>N(w)$ then the change is to $d(p_n)_{N([u,v]),N(w)}$, and we note that $(N([u,v]),N(w))>(N(u),N(v))$ by \eqref{Eq:NormUp}. If $N([u,v])<N(w)$ then the change is to $d(p_n)_{N(w),N([u,v])}$, and we note that $(N(w),N([u,v]))>(N(u),N(v))$ because $N(w)>N([u,v])>N(u)$ by \eqref{Eq:NormUp}. Thus $d(p_{n+1})<_{\mathrm{lex}} d(p_n)$ in $\mathbb N^m$.

Since $(\mathbb N^m,\lelex)$ has no infinite strictly decreasing chains, the sequence $p_1$, $p_2$, $\dots$ must terminate in finitely many steps, say at $p_n$, by reaching $d(p_n)=(0,0,\dots)$.
\end{proof}

\begin{remark}
We will need only a weaker version of Lemma \ref{Lm:Reorder} in the proof of Theorem \ref{Th:MainGroups} and later on, where we organize the product so that all factors with the same support are adjacent (in the associative case) or form a subpolynomial (in the general case). It is therefore tempting to consider a simplified norm $N(t) = \supp(t)$. However, such a choice would not yield a strict inequality in \eqref{Eq:NormUp} when $\supp(u)$ properly contains $\supp(v)$. The commutator weight must be taken into consideration for the rewriting process to stop.
\end{remark}

\begin{theorem}[\cite{AE}]\label{Th:MainGroups}
A group is $k$-nilpotent if and only if it is $k$-supernilpotent.
\end{theorem}
\begin{proof}
Let $G$ be a group. If $G$ is $k$-supernilpotent, consider the simple commutator term $t(x_1,\dots,x_{k+1})=[[\dots[[x_1,x_2],x_3]\dots],x_{k+1}]$ of weight $k+1$. Since $t$ is absorbing, it vanishes on $G$, and $G$ is therefore $k$-nilpotent.

Conversely, suppose that $G$ is $k$-nilpotent. Let $p$ be an absorbing polynomial on $G$ in variables $X=\{x_1,\dots,x_{k+1}\}$. By Lemma \ref{Lm:Reorder}, we can rewrite $p$ up to equivalence as a product $t_1\cdots t_\ell$ of complex commutator polynomials over $X^{\pm 1}\cup G$ with nondecreasing norms and weights at most $k$. Since the weights are at most $k$, $\supp(t_i)$ is a proper subset of $X$, for every $i$. For a proper subset $S$ of $\{1,\dots,k+1\}$, let $p_S$ be the product of all (consecutive) $t_i$ with $\supp(t_i)=\{x_s:s\in S\}$. Then surely $p$ is equivalent to $\prod p_S$, where the product ranges over all proper subsets $S$ of $\{1,\dots,k+1\}$.

We claim that every $p_S$ is absorbing on $G$. If $S\ne\emptyset$, every factor of $p_S$ is a complex commutator polynomial with at least one variable, hence absorbing. It follows that $1 = p(1,\dots,1)= p_\emptyset$ since $p$ is absorbing, and $p_\emptyset$ is therefore absorbing as well.

All assumptions of Proposition \ref{Pr:MainTechnical} are now satisfied and hence $p$ is constant, proving that $G$ is $k$-supernilpotent.
\end{proof}

\section{Background on loops}\label{sec:loops_lemmas}

Throughout this section, let $Q=(Q,\cdot,\rdiv,\ldiv,1)$ be a loop.

\subsection{One-sided inverses and inverses}

For every $a\in Q$ there exist unique elements $a^\lambda,a^\rho\in Q$ such that $a^\lambda a = 1 = aa^\rho$. We call $a^\lambda$ the \emph{left inverse} of $a$ and $a^\rho$ the \emph{right inverse of $a$}.
Of course, $a^\lambda = 1\rdiv a$ and $a^\rho = a\ldiv 1$. If the two one-sided inverses coincide then the (two-sided) inverse $a^\lambda = a^\rho$ will be denoted by $a^{-1}$.

If $S$ is an associative subloop of $Q$ then $S$ is a group, all elements of $S$ have two-sided inverses and we have $a\ldiv b = a^{-1}b$, $a\rdiv b = ab^{-1}$ for all $a,b\in S$.

\subsection{Commutator and associator terms}

In the variety of loops, a \emph{commutator term} is any binary term $\comm{x,y}$ such that the equivalence
\begin{displaymath}
    \comm{a,b}=1\qquad\Leftrightarrow\qquad ab=ba
\end{displaymath}
holds for all elements $a$, $b$ of every loop. Similarly, an \emph{associator term} is any ternary term $\assoc{x,y,z}$ such that the equivalence
\begin{displaymath}
    \assoc{a,b,c} = 1\qquad\Leftrightarrow\qquad a(bc)=(ab)c
\end{displaymath}
holds for all elements $a$, $b$, $c$ of every loop. Note that commutator and associator terms are absorbing, that is,
\begin{equation}\label{Eq:CAabsorbing}
    \comm{1,b} = \comm{a,1} = \comm{1,b,c} = \comm{a,1,c} = \comm{a,b,1} = 1
\end{equation}
whenever $a$, $b$, $c$ are loop elements.

We define the \emph{standard commutator term} $[x,y]$ by
\begin{displaymath}
    [x,y] =(yx)\ldiv (xy)
\end{displaymath}
and the \emph{standard associator term} $[x,y,z]$ by
\begin{displaymath}
    [x,y,z] = (x(yz))\ldiv ((xy)z).
\end{displaymath}
We will frequently use the identities $(xy)z = (x(yz))[x,y,z]$ and $xy = (yx)[x,y]$. Note that the standard commutator term in loops generalizes the standard commutator term $x^{-1}y^{-1}xy$ in groups.

A mapping $f:Q^k\to Q$ is said to be \emph{linear in the $i$th coordinate}, for some $1\le i\le k$, if for every $a_1,\dots,a_{i-1},a_{i+1},\dots,a_k\in Q$ the mapping $g:Q\to Q$ defined by $g(u)=f(a_1,\dots,a_{i-1},u,a_{i+1},\dots,a_k)$ is a homomorphism, that is, $g(uv)=g(u)g(v)$ holds for all $u,v\in Q$ (and therefore also $g(u\ldiv v) = g(u)\ldiv g(v)$ and $g(u\rdiv v) = g(u)\rdiv g(v)$ holds for all $u,v\in Q$). If $f:Q^k\to Q$ is linear in every coordinate, it is said to be \emph{multilinear}.

\subsection{The center and the nucleus}

The \emph{nucleus} of $Q$ is defined by
\begin{displaymath}
    \nuc{Q}=\{a\in Q: \assoc{a,u,v} = \assoc{u,a,v} = \assoc{u,v,a} =1\text{ for all }u,v\in Q\}.
\end{displaymath}
The nucleus is always an associative subloop but not necessarily a normal subloop of $Q$.

The \emph{center} of $Q$ is defined by
\begin{displaymath}
    Z(Q)=\{a\in Q: a\in\nuc{Q}\text{ and }\comm{a,u} = 1\text{ for all }u\in Q\}.
\end{displaymath}
The center is always a commutative, associative and normal subloop of $Q$. Elements of $Z(Q)$ are called \emph{central}.

The following lemmas will be helpful in the commutator-associator calculus of $3$-nilpotent loops.

\begin{lemma}\label{Lm:CACalculus}
Let $Q$ be a loop and $a,b,c\in Q$ such that $[a,b,c]\in\nuc Q$. Then $a(bc) = ((ab)c)[a,b,c]^{-1}$.
\end{lemma}

\begin{proof}
If $[a,b,c]\in\nuc{Q}$ then the two-sided inverse of $[a,b,c]$ exists, and from $(ab)c = (a(bc))[a,b,c]$ we deduce
\begin{displaymath}
    ((ab)c)[a,b,c]^{-1} = (a(bc))[a,b,c]\cdot [a,b,c]^{-1} = a(bc)\cdot [a,b,c][a,b,c]^{-1} = a(bc).\qedhere
\end{displaymath}
\end{proof}

\begin{lemma}\label{Lm:Useful}
Let $Q$ be a loop such that $Q/Z(Q)$ is a group. Then the following identities hold:
\begin{enumerate}
\item[(i)] $x\rdiv y = xy^\lambda[x,y^\lambda,y]^{-1}$ and $x\ldiv y = x^\rho y [x,x^\rho,y]$,
\item[(ii)] $(xy)^\lambda = y^\lambda x^\lambda[y^\lambda,x^\lambda,x]^{-1}[y^\lambda x^\lambda,x,y]$ and $(xy)^\rho = y^\rho x^\rho[x,y,y^\rho]^{-1}[xy,y^\rho,x^\rho]$,
\item[(iii)] $aa^\lambda=[a,a^\lambda]=[a^\lambda,a]^{-1}\in Z(Q)$ for every $a\in Q$,
\item[(iv)]  $x^\rho = x^\lambda [x,x^\lambda]^{-1}$,
\item[(v)] $(x^\lambda)^\lambda = x[x,x^\lambda]^{-1}$,
\item[(vi)] $[xy,(xy)^\lambda]=[x,x^\lambda][y,y^\lambda][x,y,y^\lambda][xy,y^\lambda,x^\lambda]^{-1}[y^\lambda,x^\lambda,x]^{-1}[y^\lambda x^\lambda,x,y]$.
\end{enumerate}
\end{lemma}
\begin{proof}
All associators are central since $Q/Z(Q)$ is a group, hence Lemma \ref{Lm:CACalculus} applies. For part (i), $x = x(y^\lambda y) = (xy^\lambda)y[x,y^\lambda,y]^{-1}$ implies $x\rdiv y = (xy^\lambda)[x,y^\lambda,y]^{-1}$. Dually, $y = (xx^\rho)y = x(x^\rho y)[x,x^\rho,y]$ implies $x\ldiv y = x^\rho y [x,x^\rho,y]$. Part (ii) follows from
\begin{displaymath}
    1=(y^\lambda\cdot x^\lambda x)y = (y^\lambda x^\lambda\cdot x)y[y^\lambda,x^\lambda,x]^{-1} = (y^\lambda x^\lambda)(xy)[y^\lambda,x^\lambda,x]^{-1}[y^\lambda x^\lambda,x,y]
\end{displaymath}
and from the dual identity
\begin{displaymath}
    1 = (x\cdot yy^\rho)\cdot x^\rho = (xy\cdot y^\rho)x^\rho[x,y,y^\rho]^{-1} = (xy)(y^\rho x^\rho)[x,y,y^\rho]^{-1}[xy,y^\rho,x^\rho].
\end{displaymath}

The group $Q/Z(Q)$ has two-sided inverses. For every $a\in Q$, we then have $a^\lambda Z(Q) = (aZ(Q))^\lambda = (aZ(Q))^\rho = a^\rho Z(Q)$ and therefore $a^\lambda\ldiv a^\rho\in Z(Q)$. Now,
\begin{displaymath}
    xx^\rho = 1 = x^\lambda x = (xx^\lambda)[x^\lambda,x] = x\cdot x^\lambda[x^\lambda,x][x,x^\lambda,[x^\lambda,x]],
\end{displaymath}
so $x^\rho = x^\lambda[x^\lambda,x][x,x^\lambda,[x^\lambda,x]]$ and thus
\begin{equation}\label{Eq:UsefulAux}
    x^\lambda \ldiv x^\rho = [x^\lambda,x][x,x^\lambda,[x^\lambda,x]].
\end{equation}
We have $aa^\lambda = (a^\lambda a)[a,a^\lambda] = [a,a^\lambda]$ and $1=a^\lambda a = (aa^\lambda)[a^\lambda,a]$, so $[a^\lambda,a] = (aa^\lambda)\ldiv 1$.
Since $a^\lambda\ldiv a^\rho$ is central and all associators are central, it follows from \eqref{Eq:UsefulAux} that $[a^\lambda,a]=(aa^\lambda)\ldiv 1$ is central, thus also that $aa^\lambda$ is central and $aa^\lambda=[a^\lambda,a]^{-1}$, finishing (iii). Then the associator $[x,x^\lambda,[x^\lambda,x]]$ vanishes in \eqref{Eq:UsefulAux} and thus $x^\rho = x^\lambda[x^\lambda,x] = x^\lambda[x,x^\lambda]^{-1}$, proving (iv).

For (v), $(x^\lambda)^\lambda x^\lambda = 1 = x^\lambda x = xx^\lambda[x^\lambda,x]$ yields $(x^\lambda)^\lambda = x[x^\lambda,x] = x[x,x^\lambda]^{-1}$. Finally,
\begin{align*}
    [xy,(xy)^\lambda] &=(xy)(xy)^\lambda=(xy)(y^\lambda x^\lambda) [y^\lambda,x^\lambda,x]^{-1}[y^\lambda x^\lambda,x,y],\\
    (xy)(y^\lambda x^\lambda) &= (xy\cdot y^\lambda)x^\lambda[xy,y^\lambda,x^\lambda]^{-1},\\
    (xy\cdot y^\lambda)x^\lambda &= (x\cdot yy^\lambda)x^\lambda[x,y,y^\lambda],\\
    (x\cdot yy^\lambda)x^\lambda & = (xx^\lambda)(yy^\lambda) = [x,x^\lambda][y,y^\lambda],
\end{align*}
using $yy^\lambda\in Z(Q)$ in the last equality. We get (vi) by combining the last four identities.
\end{proof}

\begin{lemma}\label{Lm:LambdaPowers}
Let $Q$ be a loop such that $Q/Z(Q)$ is a group. For $a\in Q$ let $f(a)=a^\lambda$. Then for any integer $n\ge 0$ we have
\begin{displaymath}
    f^n(a) = \left\{\begin{array}{ll}
        a[a,a^\lambda]^{-n/2},&\text{ if $n$ is even},\\
        a^\lambda[a,a^\lambda]^{(n-1)/2},&\text{ if $n$ is odd}.
    \end{array}\right.
\end{displaymath}
\end{lemma}
\begin{proof}
We will use Lemma \ref{Lm:Useful} throughout. First note that the powers $[a,a^\lambda]^m$ are well-defined since $[a,a^\lambda]$ is central. For any central element $c$ we have $c^\lambda = c^{-1}$.
Now, $f^0(a)=a$, $f^1(a)=a^\lambda$ and $f^2(a)=a[a,a^\lambda]^{-1}$, in accordance with the formula. Since central elements can be omitted from associators, the formula $(ab)^\lambda = b^\lambda a^\lambda[b^\lambda,a^\lambda,a]^{-1}[b^\lambda a^\lambda,a,b]$ reduces to $(ab)^\lambda = a^\lambda b^{-1}$ when $b$ is central. Therefore, if $n$ is even and $f^n(a)=a[a,a^\lambda]^{-n/2}$ then $f^{n+1}(a) = (a[a,a^\lambda]^{-n/2})^\lambda = a^\lambda [a,a^\lambda]^{n/2}$. Similarly, if $n$ is odd and $f^n(a) = a^\lambda[a,a^\lambda]^{(n-1)/2}$ then $f^{n+1}(a) = (a^\lambda)^\lambda [a,a^\lambda]^{-(n-1)/2} = a[a,a^\lambda]^{-1}[a,a^\lambda]^{-(n-1)/2} = a[a,a^\lambda]^{-(n+1)/2}$.
\end{proof}

\section{Equational bases for supernilpotence in loops}\label{sec:sn in loops}

In this section we find an equational basis for $k$-supernilpotent loops for $k\in\{1,2,3\}$.

\subsection{Supernilpotent loops of class 1 and 2}

\begin{proposition}\label{Pr:SupernilpotentLoops12}
A loop is $1$-supernilpotent if and only if it is an abelian group. A loop is $2$-supernilpotent if and only if it is a $2$-nilpotent group.
\end{proposition}
\begin{proof}
Recall that commutator and associator terms are absorbing in loops. Therefore, $2$-supernilpotent loops are groups and $1$-supernilpotent loops are abelian groups. The converse implications follow from Theorem \ref{Th:MainGroups}.
\end{proof}

\begin{example}\label{Ex:6}
This example shows that $2$-nilpotent loops need not be $k$-supernilpotent, for any $k$. The following loop $Q$ of order $6$ has center $Z(Q)=\{1,4\}$ and is $2$-nilpotent. But it cannot be supernilpotent since it is directly indecomposable (there is actually no subloop of order 3), thus violating \cite[Lemma 7.6]{AM}.
\begin{displaymath}
\begin{array}{r|rrrrrr}
    Q&1&2&3&4&5&6\\
    \hline
    1&1&2&3&4&5&6\\
    2&2&3&1&5&6&4\\
    3&3&1&5&6&4&2\\
    4&4&5&6&1&2&3\\
    5&5&6&4&2&3&1\\
    6&6&4&2&3&1&5
\end{array}
\end{displaymath}
\end{example}

\subsection{Commutator and associator identities for 3-supernilpotent loops}

Fix a commutator term $c = \comm{\_,\_}$ and an associator term $a = \assoc{\_,\_,\_}$. Let $Y$ be the set of variables, their left inverses (i.e., symbols $x^\lambda$) and constants. \emph{Complex commutator/associator polynomials} in loops and their \emph{weights} with respect to $Y$ are defined in analogy with complex commutator polynomials in groups, with the added rule that if $t_1$, $t_2$, $t_3$ are complex commutator/associator polynomials of weight $k_1$, $k_2$, $k_3$, respectively, then $\assoc{t_1,t_2,t_3}$ is a complex commutator/associator polynomial of weight $k_1+k_2+k_3$.

Let $\mathcal I_{c,a}$ consist of the following identities:
\begin{align}
    1&=\comm{x,\assoc{y,u,v}},\label{Eq:AZ1} \\
    1&=\assoc{x,y,\assoc{u,v,w}} = \assoc{x,\assoc{u,v,w},y} = \assoc{\assoc{u,v,w},x,y},\label{Eq:AZ2}\\
    1&=\assoc{x,y,\comm{u,v}}=\assoc{x,\comm{u,v},y}=\assoc{\comm{u,v},x,y},\label{Eq:CN}\\
    1&=\comm{x,\comm{y,\comm{u,v}}} = \comm{x,\comm{\comm{u,v},y}},\label{Eq:CCCtrivial1}\\
    1&=\comm{\comm{y,\comm{u,v}},x} = \comm{\comm{\comm{u,v},y},x},\label{Eq:CCCtrivial2}\\
    1&=\comm{\comm{x,y},\comm{u,v}},\label{Eq:CCcommute}\\
    \assoc{xy,u,v}&=\assoc{x,u,v}\,\assoc{y,u,v},\label{Eq:AM1}\\
    \assoc{u,xy,v}&=\assoc{u,x,v}\,\assoc{u,y,v},\label{Eq:AM2}\\
    \assoc{u,v,xy}&=\assoc{u,v,x}\,\assoc{u,v,y}.\label{Eq:AM3}
\end{align}
Note that the identities \eqref{Eq:AZ1}--\eqref{Eq:AZ2} say that associators are in the center, \eqref{Eq:CN} says that commutators are in the nucleus, \eqref{Eq:CCCtrivial1}--\eqref{Eq:CCcommute} say that all complex commutators of weight $4$ vanish, and \eqref{Eq:AM1}--\eqref{Eq:AM3} say that associators are multilinear.

\begin{proposition}\label{Pr:3snImpliesI}
For any choice of the commutator term $c$ and the associator term $a$, every $3$-supernilpotent loop satisfies the identities $\mathcal I_{c,a}$.
\end{proposition}
\begin{proof}
Recall the identity \eqref{Eq:CAabsorbing} that shows that every commutator and associator is absorbing. For each identity of the form $t=1$ from among \eqref{Eq:AZ1}--\eqref{Eq:CCcommute}, we can check that $t$ is an absorbing term in at least $4$ variables, therefore constant and equal to $1$. For each identity of the form $t=s$ from among \eqref{Eq:AM1}--\eqref{Eq:AM3}, observe that $t/s$ is an absorbing term in at least $4$ variables and therefore it is constant, equal to $1$. For instance, to verify \eqref{Eq:AM1}, we form the term $\assoc{xy,u,v}\rdiv( \assoc{x,u,v}\,\assoc{y,u,v} )$ and note that with $x=1$ it reduces to $\assoc{y,u,v}\rdiv (1\cdot\assoc{y,u,v}) = 1$, similarly with $y=1$, and with $u=1$ or $v=1$ it reduces to $1\rdiv (1\cdot 1)=1$.
\end{proof}

\subsection{Supernilpotent loops of class 3}

For a set $X$, let $X^\lambda = \{x^\lambda:x\in X\}$. We could formulate the next result more generally but it will suffice for our purposes as stated.

\begin{lemma}\label{Lm:SimplifyAssoc}
Let $Q$ be a loop satisfying the identities $\eqref{Eq:AZ1}$, $\eqref{Eq:AZ2}$ and $\eqref{Eq:AM1}$--$\eqref{Eq:AM3}$ for the standard commutator $c$ and the standard associator $a$. Let $p,q,r$ be polynomials on $Q$ with variables from $X$ such that the only operations appearing in $p,q,r$ are the multiplication and the left inverse. Then:
\begin{enumerate}
	\item[(i)] The polynomial $[p,q,r]$ is equivalent to a product of polynomials of the form $[u,v,w]^{\pm1}$ where $u,v,w\in X\cup X^\lambda\cup Q$.
	\item[(ii)] The polynomial $[p,p^\lambda]$ is equivalent to a product of polynomials of the form $[u,v,w]^{\pm1}$ and $[u,u^\lambda]^{\pm1}$ where $u,v,w\in X\cup X^\lambda\cup Q$.
\end{enumerate}
\end{lemma}

\begin{proof}
Since \eqref{Eq:AZ1} and \eqref{Eq:AZ2} hold, all associators are central and $Q/Z(Q)$ is a group. Then all commutators of the form $[u,u^\lambda]$ are central by Lemma \ref{Lm:Useful}. All central elements can be ignored within commutators and associators. For any polynomial $p$ that is not a constant or a variable, call the operation at the root of $p$ the \emph{root operation}, here either the multiplication or the left inverse. For associators $[p,q,r]$, we apply recursively the identities $\eqref{Eq:AM1}$--$\eqref{Eq:AM3}$ (when the root operation of interest is the multiplication), or Lemma \ref{Lm:Useful}(ii) (when the root operation is the left inverse followed by the multiplication), or Lemma \ref{Lm:LambdaPowers} (when the root operation is an iterated left inverse). For commutators $[p,p^\lambda]$ we proceed similarly, except that we apply Lemma \ref{Lm:Useful}(vi) when the root operation of $p$ is the multiplication.
\end{proof}

\begin{lemma}\label{Lm:FirstRewrite}
Let $Q$ be a loop satisfying the identities $\eqref{Eq:AZ1}$, $\eqref{Eq:AZ2}$ and $\eqref{Eq:AM1}$--$\eqref{Eq:AM3}$ for the standard commutator $c$ and the standard associator $a$. Let $p$ be a polynomial on $Q$ with variables from $X$. Then $p$ is equivalent to a polynomial $rs$, where $r$ is a product of elements from $X\cup X^\lambda\cup Q$, and $s$ is a product of polynomials of the form $[u,v,w]^{\pm 1}$ and $[u,u^\lambda]^{\pm1}$ where $u,v,w\in X\cup X^\lambda\cup Q$.
\end{lemma}
\begin{proof}
We reach the desired expression in four stages. Since $Q/Z(Q)$ is a group, all elements $[a,b,c]$ and $[a,a^\lambda]$ are central. Within this proof, call a polynomial $\emph{ca-free}$ it if contains no commutator/associator polynomials of weight bigger than $1$, and a $\emph{ca-product}$ if it is a product of polynomials of the form $[u,v,w]^{\pm 1}$ and $[u,u^\lambda]^{\pm 1}$ (with $u$, $v$, $w$ not necessarily in $X\cup X^\lambda\cup Q$).

In stage 1, we replace all divisions in the ca-free polynomial $p$ as follows. By Lemma \ref{Lm:Useful} and by multilinearity of associators, we have the identities $x\rdiv y = xy^\lambda[x,y^\lambda,y]^{-1}$ and
\begin{displaymath}
    x\ldiv y = x^\rho y [x,x^\rho,y] = x^\lambda [x,x^\lambda]^{-1} y [x,x^\lambda [x,x^\lambda]^{-1},y] = x^\lambda y [x,x^\lambda]^{-1} [x,x^\lambda,y].
\end{displaymath}
Moving from the leaves to the root of the tree corresponding to $p$, at every node we replace $u\rdiv v$ by $uv^\lambda[u,v^\lambda,v]^{-1}$ and every $u\ldiv v$ by $u^\lambda v [u,u^\lambda]^{-1} [u,u^\lambda,v]$. By moving central elements to the right, we obtain a polynomial $p_1s_1$ equivalent to $p$ which contains no divisions and right inverses, $p_1$ is ca-free and $s_1$ is a ca-product.

In stage 2, we use the identity $(xy)^\lambda = y^\lambda x^\lambda[y^\lambda,x^\lambda,x]^{-1}[y^\lambda x^\lambda,x,y]$ on $p_1$ iteratively to obtain a polynomial $p_2s_2$ equivalent to $p_1$ in which the left inverse is never applied to a product (that is, it appears only in subterms $(((x^\lambda)^\lambda)\cdots)^\lambda$ with $x\in X$), $p_2$ is ca-free and $s_2$ is a ca-product.

In stage 3, we use Lemma \ref{Lm:LambdaPowers} on $p_2$, replacing $(((x^\lambda)^\lambda)\cdots)^\lambda=f^n(x)$ with $x[x,x^\lambda]^{-n/2}$ or with $x^\lambda[x,x^\lambda]^{(n-1)/2}$, depending on the parity of $n$. We obtain a polynomial
$p_3s_3$ equivalent to $p_2$ in which the left inverse operation is never iterated, $p_3$ is ca-free and $s_3$ is a ca-product.

Note that $r=p_3$ is now a product of elements from $X\cup X^\lambda\cup Q$. Moreover, $p$ is equivalent to $rs_3s_2s_1$ and $s_3s_2s_1$ is a ca-product that contains no divisions and right inverses. In stage $4$ we use Lemma \ref{Lm:SimplifyAssoc} to replace $s_3s_2s_1$ with an equivalent polynomial $s$ that is a product of polynomials of the form $[u,v,w]^{\pm 1}$ and $[u,u^\lambda]^{\pm1}$ with $u,v,w\in X\cup X^\lambda\cup Q$, finishing the proof.
\end{proof}

We now establish the analog of Lemma \ref{Lm:Reorder} for $3$-supernilpotent loops, paying attention also to the way in which products can be associated.

\begin{lemma}\label{Lm:LoopReorder}
Let $Q$ be a loop satisfying the identities $\mathcal I_{c,a}$ for the standard commutator term $c$ and the standard associator term $a$. Let $p$ be a polynomial on $Q$ with support $X$. Then $p$ is equivalent to a product of complex commutator/associator polynomials in $X\cup X^\lambda\cup Q$ such that the weight of every factor is at most $3$ and all factors with the same support form a subpolynomial of the product.
\end{lemma}
\begin{proof}
We can assume $p=rs$ as in Lemma \ref{Lm:FirstRewrite}. Since all associators are central and multilinear, we can reassociate the factors of $r$ at will, collecting the associators that arise in $s$. Once some commutators arise in $r$, we can ignore them in associators thanks to the identity \eqref{Eq:CN}. The combined effect of the identities \eqref{Eq:CCCtrivial1}--\eqref{Eq:CCcommute} is that all complex commutators of weight $4$ or higher vanish, and we can thus ignore the corresponding complex commutator polynomials since we work up to equivalence. We can therefore think of rewriting $r$ as if it were a product of elements in a $3$-nilpotent group. By Lemma \ref{Lm:Reorder}, we can rewrite $r$ as a product of complex commutators with nondecreasing norm and weight at most $3$. In particular, all factors of $r$ with the same support will be automatically ``adjacent'', forming a subpolynomial of $r$. Finally, the central factors comprising $s$ can be moved to their desired position for free.
\end{proof}

Here is the main result:

\begin{theorem}\label{Th:MainLoops}
The following conditions are equivalent for a loop $Q$:
\begin{enumerate}
\item[(i)] $Q$ is $3$-supernilpotent,
\item[(ii)] $Q$ satisfies the identities $\mathcal I_{c,a}$ with any choice of commutator term $c$ and associator term $a$,
\item[(iii)] $Q$ satisfies the identities $\mathcal I_{c,a}$ with the standard commutator term $c$ and the standard associator term $a$.
\end{enumerate}
\end{theorem}
\begin{proof}
By Proposition \ref{Pr:3snImpliesI}, every $3$-supernilpotent loop satisfies the identities $\mathcal I_{c,a}$, so (i) implies (ii). Clearly, (ii) implies (iii). For the rest of the proof suppose that $Q$ is a loop satisfying $\mathcal I_{c,a}$ with the standard commutator term $c$ and the standard associator term $a$.

Let $p$ be an absorbing polynomial on $Q$ with support $X=\{x_1,\dots,x_4\}$. Up to equivalence, we can assume by Lemma \ref{Lm:LoopReorder} that $p$ is a product of complex commutator/associator polynomials in $X\cup X^\lambda\cup Q$ of weight at most $3$ such that all factors with the same support form a subpolynomial of $p$. We can therefore write $p=\prod_S p_S$, where the product ranges over proper subsets of $\{1,2,3,4\}$ in some order and where every factor of $p_S$ is a complex commutator/associator polynomial with support $\{x_i:i\in S\}$.

We claim that every $p_S$ is absorbing on $Q$. If $S\ne\emptyset$, every factor of $p_S$ is a complex commutator/associator polynomial containing a variable, hence absorbing. It follows that $1 = p(1,\dots,1)= p_\emptyset$ since $p$ is absorbing, and $p_\emptyset$ is therefore absorbing as well.

All assumptions of Proposition \ref{Pr:MainTechnical} are now satisfied and hence $p$ is constant, proving that $Q$ is $3$-supernilpotent.
\end{proof}

\section{Consequences of and connections to 3-supernilpotence}\label{Sc:Consequences}

\subsection{Inner mappings in 3-supernilpotent loops}\label{Ss:Inner}

For a loop $Q$, let $\mlt{Q}=\langle L_a,R_a:a\in Q\rangle$ be the \emph{multiplication group} of $Q$ and $\inn{Q}=\{f\in\mlt{Q}:f(1)=1\}$ the \emph{inner mapping group} of $Q$. Consider the inner mappings $L_{a,b}$, $R_{a,b}$, $M_{a,b}$, $T_a$ and $T_{a,b}$ defined by
\begin{align*}
    L_{a,b}(u) &= (ab)\ldiv (a(bu)),\\
    R_{a,b}(u) &= ((ua)b)\rdiv (ab),\\
    M_{a,b}(u) &= a\ldiv((a(ub))/b),\\
    T_a(u) &= a\ldiv (ua),\\
    T_{a,b}(u) &= [T_a,T_b](u) = T_a^{-1}T_b^{-1}T_aT_b(u).
\end{align*}
It is well-known that $\inn{Q}=\langle L_{a,b},R_{a,b},T_a:a,b\in Q\rangle$ and it is easy to check that $c\in\nuc{Q}$ if and only if $L_{a,b}(c)=R_{a,b}(c)=M_{a,b}(c)=c$ for every $a,b\in Q$.

Analogous notation will be used also for terms. For instance, $L_{x,y}(z)$ denotes the term $(xy)\ldiv (x(yz))$.

Note that $U_{a,b}(1)=1$ for all $U\in\{L,R,M,T\}$ and all $a,b\in Q$ since $U_{a,b}\in\inn{Q}$. Moreover, $T_1=1$ and $U_{a,b}=1$ whenever $U\in\{L,R,M,T\}$ and $a=1$ or $b=1$. We will use these facts in the proof of the next result.

\begin{proposition}\label{Pr:Inn}
Let $Q$ be a $3$-supernilpotent loop. Then:
\begin{enumerate}
\item[(i)] $L_{a,b}$, $R_{a,b}$, $M_{a,b}$ and $T_{a,b}$ are automorphisms of $Q$ for every $a,b\in Q$.
\item[(ii)] $L_{a,b}$, $R_{a,b}$, $M_{a,b}$ and $T_{a,b}$ are in the center of $\inn Q$ for every $a,b\in Q$.
\item[(iii)] $T_a(bc) =T_a(b)T_a(c)$ if $a\in Q$ and at least one of $b,c$ is a commutator or an associator in $Q$.
\end{enumerate}
\end{proposition}

\begin{proof}
Let $U\in\{L,R,M,T\}$. For (i), note that $U_{a,b}(cd)=U_{a,b}(c)U_{a,b}(d)$ whenever one of $a,b,c,d$ is equal to $1$. By $3$-supernilpotence, $U_{a,b}(cd)=U_{a,b}(c)U_{a,b}(d)$ holds for all $a,b,c,d\in Q$. For (ii), it suffices to show that $U_{a,b}$ commutes with $T_c$, $L_{c,d}$ and $R_{c,d}$. We have $U_{a,b}T_c(d)=T_cU_{a,b}(d)$ whenever one of $a,b,c,d$ is equal to $1$, and similarly for $L_{c,d}$ and $R_{c,d}$. For (iii), $T_a(b\comm{c,d}) = T_a(b)T_a(\comm{c,d})$ whenever one of $a,b,c,d$ is equal to $1$. Similarly for $T_a(\comm{c,d}b) = T_a(\comm{c,d})T_a(b)$, $T_a(b\assoc{c,d,e}) = T_a(b)T_a(\assoc{c,d,e})$ and $T_a(\assoc{c,d,e}b) = T_a(\assoc{c,d,e})T_a(b)$.
\end{proof}

\begin{corollary}\label{Cr:Inn}
Let $Q$ be a 3-supernilpotent loop. Then
\begin{displaymath}
    \langle L_{a,b},M_{a,b},R_{a,b},T_{a,b}: a,b\in Q\rangle\leq Z(\inn Q)\cap\aut Q.
\end{displaymath}
\end{corollary}

\begin{remark}
A loop $Q$ is said to be \emph{automorphic} (or an \emph{A-loop}) if $\inn{Q}\le\aut{Q}$ \cite{BP, KKPV}. A loop $Q$ is said to be an \emph{AIM} loop if $\inn{Q}$ is an abelian group \cite{KVV}. Corollary \ref{Cr:Inn} makes it clear that $3$-supernilpotent loops are very close to both automorphic loops and AIM loops. In more detail, a $3$-supernilpotent loop $Q$ is automorphic if and only if $T_a\in\aut{Q}$ for every $a\in Q$. A $3$-supernilpotent loop $Q$ is an AIM loop if and only if $T_{a,b}=1$ for every $a,b\in Q$.
\end{remark}

\subsection{2-nilpotent Moufang loops}\label{Ss:Moufang}

Note that the identities \eqref{Eq:AZ1}--\eqref{Eq:CCcommute} are satisfied in every $2$-nilpotent loop. Theorem \ref{Th:MainLoops} therefore implies:

\begin{corollary}\label{Cr:2nilp}
Let $Q$ be a $2$-nilpotent loop. Then $Q$ is $3$-supernilpotent if and only if the standard associator term is multilinear.
\end{corollary}

A loop is \emph{Moufang} if it satisfies the identity $x(y(xz)) = ((xy)x)z$. By Moufang's Theorem \cite{Moufang}, Moufang loops are \emph{diassociative}, that is, every two elements generate an associative subloop.

The following corollary of Theorem \ref{Th:MainLoops} was suggested to us by Michael Kinyon.

\begin{corollary}\label{Cr:Moufang}
2-nilpotent Moufang loops are 3-supernilpotent.
\end{corollary}
\begin{proof}
Let $Q$ be a $2$-nilpotent Moufang loop. In view of Corollary \ref{Cr:2nilp}, it suffices to show that the standard associator term is multilinear. Using diassociativity and the fact that all associators are central, we get
\begin{displaymath}
    L_{a,b}(c) = (ab)^{-1}(a(bc)) = (ab)^{-1}((ab)c)[a,b,c]^{-1} = c[a,b,c]^{-1}.
\end{displaymath}
By \cite[Lemma VII.2.2]{Br-book}, the left inner mapping $L_{a,b}$ of any Moufang loop is a so-called right pseudo-automorphism of $Q$ with companion $[b,a]$, i.e.,
\begin{displaymath}
    L_{a,b}(c)(L_{a,b}(d)\cdot[b,a]) = L_{a,b}(cd)\cdot[b,a]
\end{displaymath}
for every $a,b,c,d\in Q$. Since $[b,a]\in Z(Q)$ here, it follows that $L_{a,b}$ is an automorphism of $Q$. Therefore,
\begin{displaymath}
    (cd)[a,b,cd]^{-1} = L_{a,b}(cd) = L_{a,b}(c)\cdot L_{a,b}(d) = c[a,b,c]^{-1}\cdot d[a,b,d]^{-1}
\end{displaymath}
for every $a,b,c,d\in Q$. Since associators are central, we can cancel $c$ and $d$ and deduce linearity in the first slot. By \cite[Lemma VII.5.5]{Br-book}, if a Moufang loop satisfies the identity $[[x,y,z],x]=1$ (which certainly holds here) then it also satisfies the identity $[x,y,z]=[y,z,x]$. Multilinearity now follows from linearity in the first slot.
\end{proof}

\begin{remark}
A \emph{code loop} is a Moufang loop $Q$ possessing a central subloop $Z$ of order 2 such that $Q/Z$ is an elementary abelian 2-group, cf. \cite[Chapter 4]{Asch}, \cite{DV}, \cite{Gr} and \cite{OBV}. Code loops play an important role in the construction of the Monster group \cite{Conway}. Since code loops are Moufang and $2$-nilpotent by definition, they are $3$-supernilpotent by Corollary \ref{Cr:Moufang}. It follows that the standard associator term is multilinear in code loops, a well-known key property of code loops.
\end{remark}

\subsection{Supernilpotence modulo the center and modulo the nucleus}\label{Ss:ZN}

We show that the class of supernilpotence is reduced by at least one modulo the center and by at least two modulo the nucleus, provided that the nucleus is a normal subloop. We also show that the nucleus is normal in $3$-supernilpotent loops.

\begin{lemma}[\cite{SS}]\label{Lm:Q/Z}
Let $Q$ be a $k$-supernilpotent loop and $k\geq 2$. Then $Q/Z(Q)$ is $(k-1)$-supernilpotent.
\end{lemma}
\begin{proof}
Let $p$ be a $k$-ary polynomial absorbing on $Q/Z(Q)$, i.e., $p(a_1,\dots,a_k)\in Z(Q)$ whenever $a_1,\dots,a_k\in Q$ and $a_i\in Z(Q)$ for some $i$.
Define new polynomials
\begin{align*}
    q(y,x_1,\dots,x_k) &= [y,p(x_1,\dots,x_k)],\\
    r_1(y,z,x_1,\dots,x_k) &= [y,z,p(x_1,\dots,x_k)],\\
    r_2(y,z,x_1,\dots,x_k) &= [y,p(x_1,\dots,x_k),z],\\
    r_3(y,z,x_1,\dots,x_k) &= [p(x_1,\dots,x_k),y,z].
\end{align*}
Since all $q$, $r_1$, $r_2$, $r_3$ are absorbing on $Q$ and of arity bigger than $k$, they are constant (equal to $1$) on $Q$. Therefore $p(a_1,\dots,a_k)\in Z(Q)$ for all $a_1,\dots,a_k\in Q$ and hence
$p$ is constant on $Q/Z(Q)$.
\end{proof}

\begin{lemma}\label{Lm:Q/N}
Let $Q$ be a $k$-supernilpotent loop and $k\geq 3$. Suppose that $\nuc{Q}$ is normal in $Q$. Then $Q/\nuc{Q}$ is $(k-2)$-supernilpotent.
\end{lemma}
\begin{proof}
The proof is analogous to the proof of Lemma \ref{Lm:Q/Z}, omitting the polynomial~$q$.
\end{proof}

\begin{proposition}\label{Pr:Factors}
Let $Q$ be a 3-supernilpotent loop. Then:
\begin{enumerate}
\item[(i)] $\nuc{Q}$ is a normal subloop of $Q$.
\item[(ii)] $Q/\nuc{Q}$ is an abelian group and $Q/Z(Q)$ is a $2$-nilpotent group.
\end{enumerate}
\end{proposition}
\begin{proof}
(i) It suffices to show that the generators $L_{a,b}$, $R_{a,b}$ and $T_c$ of $\inn{Q}$ map $\nuc{Q}$ into $\nuc{Q}$. Recall that $d\in\nuc{Q}$ if and only if $L_{a,b}(d)=R_{a,b}(d)=M_{a,b}(d)=d$ for all $a,b\in Q$. Let $d\in\nuc{Q}$. By Proposition \ref{Pr:Inn}, $L_{a,b}T_c(d)=T_cL_{a,b}(d)=T_c(d)$. Similarly, $R_{a,b}T_c(d) = T_c(d) = M_{a,b}T_c(d)$. Hence $T_c(d)\in\nuc{Q}$.

(ii) By Lemma \ref{Lm:Q/N}, $Q/\nuc{Q}$ is $1$-supernilpotent, so an abelian group by Proposition \ref{Pr:SupernilpotentLoops12}. By Lemma \ref{Lm:Q/Z}, $Q/Z(Q)$ is $2$-supernilpotent, so a $2$-nilpotent group by Proposition  \ref{Pr:SupernilpotentLoops12}.
\end{proof}

\begin{remark}
The nilpotent loop $Q$ in Example \ref{Ex:6} satisfies $\nuc{Q}=Z(Q)\unlhd Q$ and the factor $Q/\nuc{Q} = Q/Z(Q)$ is the cyclic group of order $3$. But $Q$ is not supernilpotent. Hence the conditions of Proposition \ref{Pr:Factors} are not sufficient for supernilpotence, much less for $3$-supernilpotence.
\end{remark}

\subsection{Computational results: 3-supernilpotence in small nilpotent loops}\label{Ss:Comp}

All nilpotent loops of orders $8$ and $9$ can be found in the \texttt{GAP} \cite{GAP} package \texttt{LOOPS} \cite{LOOPS} by calling \texttt{NilpotentLoop(n,m)}.

Semani\v{s}inov\'a implemented an algorithm in \cite{Sema} for deciding whether small nilpotent loops are $3$-supernilpotent. The algorithm, based on the so-called forks developed theoretically by Opr\v{s}al \cite{Op}, is demanding both in time and in memory. It did not finish in quite a few cases. Of the $134$ nilpotent loops of order $8$, Semani\v{s}inov\'a was able to determine the status of $3$-supernilpotence for all but $28$ loops. The algorithm did not finish on any of the $8$ nilpotent loops of order $9$, cf. \cite[Tables 3.1, 3.2]{Sema}.

Our main result, Theorem \ref{Th:MainLoops}, can be turned into an effective test of $3$-supernilpotence for small loops. It takes a few seconds to determine $3$-supernilpotence for all nilpotent loops of order $8$ and $9$. We verified all results of Semani\v{s}inov\'a. In addition, in all of the $28$ previously undetermined cases of nilpotent loops of order $8$, the loops are actually $3$-supernilpotent. For the eight nilpotent loops of order $9$, \texttt{NilpotentLoop(9,m)} is $3$-supernilpotent if and only if $\texttt{m}\in\{4,6,8\}$.

\end{document}